\documentclass[12pt]{amsart}
\usepackage{amsmath}
\usepackage{amscd}
\usepackage{amsfonts}
\usepackage{amssymb}
\usepackage{color}
\usepackage{fullpage}
\usepackage{mathtools}
\usepackage[hypertexnames=false, colorlinks, citecolor=red,linkcolor=blue, urlcolor=black]{hyperref}
\usepackage[all]{xy}
\usepackage{tikz-cd}
\usepackage[capitalise]{cleveref}
\usepackage{enumitem}
\numberwithin{equation}{section}

\newtheorem{theorem}{Theorem}[section]
\newtheorem{lemma}[theorem]{Lemma}

\newtheorem{proposition}[theorem]{Proposition}

\theoremstyle{definition}

\newtheorem{remark}[theorem]{Remark}

\newcommand{\Stab}{\operatorname{Stab}}
\newcommand{\Aut}{\operatorname{Aut}}

\newcommand{\PGL}{\operatorname{PGL}}

\newcommand{\GL}{\operatorname{GL}}

\newcommand{\Sym}{\operatorname{S}}
\newcommand{\Alt}{\operatorname{A}}
\newcommand{\ed}{\operatorname{ed}}

\newcommand{\Char}{\operatorname{char}}

\newcommand{\Tr}{\operatorname{Tr}}

\begin{document}

\author{Oakley Edens}   
% \email{oedens@shaw.ca}
\thanks{Oakley Edens was partially supported by an Undergraduate Student Research Award (USRA) from 
the National Sciences and Engineering Research Council of Canada.} 

\author{Zinovy Reichstein}
\address{Department of Mathematics\\
	% 1984 Mathematics Road\\
	University of British Columbia\\
	Vancouver, BC V6T 1Z2\\Canada}
% \email{reichst@math.ubc.ca}
\thanks{Zinovy Reichstein was partially supported by an Individual Discovery Grant from the
	National Sciences and Engineering Research Council of
	Canada.}


%%%%%%%%%%%%%%%%%%%%%%%%%%%%%%%%%%%%%%%%
% 12E05(1973–now) Polynomials in general fields (irreducibility, etc.)
% 14G17(2010–now) Positive characteristic ground fields in algebraic geometry
% 14L30  	Group actions on varieties or schemes (quotients) [See also 13A50, 14L24, 14M17]
% 14E05  	Rational and birational maps
%%%%%%%%%%%%%%%%%%%%%%%%%%%%%%%%%%%%%%%%
\title{Essential dimension of symmetric groups in prime characteristic}

\maketitle
\begin{abstract} The essential dimension $\ed_k(\Sym_n)$ of the symmetric group $\Sym_n$ is the minimal integer $d$ such that 
the general polynomial $x^n + a_1 x^{n-1} + \ldots + a_n$ can be reduced to
a $d$-parameter form by a Tschirnhaus transformation. Finding this number is a long-standing open problem, 
originating in the work of Felix Klein, long before essential dimension of a finite group was formally defined.  
% This problem was studied in classical literature under the name of ``the Klein resolvent problem".
We now know that $\ed_k(\Sym_n)$ lies between $\lfloor n/2 \rfloor$ and $n-3$
for every $n \geqslant 5$ and every field $k$ of characteristic different from $2$. Moreover, if $\Char(k) = 0$, then
$\ed_k(\Sym_n) \geqslant \lfloor (n+1)/2 \rfloor$ for any $n \geqslant 6$. 
The value of $\ed_k(\Sym_n)$ is not known for any $n \geqslant 8$ and any field $k$, though 
it is widely believed that $\ed_k(\Sym_n)$ should be $n-3$ for every $n \geqslant 5$, 
at least in characteristic $0$. In this paper we show that for every odd prime $p$
there are infinitely many positive integers $n$ such that $\ed_{\mathbb F_p}(\Sym_n) \leqslant n-4$.
\end{abstract}

\section{Introduction}

The essential dimension $\ed_k(\Sym_n)$ of the symmetric group $\Sym_n$ is the smallest integer $d$ such that 
the general polynomial $x^n + a_1 x^{n-1} + \ldots + a_n$ can be reduced to
a $d$-parameter form by a Tschirnhaus transformation. The geometric definition of essential dimension and some background material 
can be found in Section~\ref{sect.prelim}; for a comprehensive overview, see~\cite{merkurjev-survey, icm}.

Finding $\ed_k(\Sym_n)$ is a long-standing open problem, which goes back to F.~Klein~\cite{klein};
% and was studied in the classical literature under 
% the name of ``the Klein resolvent problem"; see, e.g., 
cf.~also N.~Chebotarev~\cite{chebotarev}.
Essential dimension of a finite group was formally defined by J. Buhler and the second authors
in~\cite{br}, where the inequalities
\begin{equation} \label{e.br}
\ed_k(\Sym_n) \geqslant \lfloor n/2 \rfloor \quad \text{($n \geqslant 5$)} \quad \quad \text{and} \quad \quad
\ed_k(\Sym_n) \leqslant n - 3 
\end{equation} 
were proved.
The field $k$ was assumed to be of characteristic $0$ in~\cite{br}, but the proof of the first inequality
in~\eqref{e.br} given there goes through for any field $k$ of characteristic different from $2$.
The second inequality is valid over an arbitrary field $k$. 
% (see \cite[Theorem 6.5(c)]{br}) 
A.~Duncan~\cite{duncan2010essential} subsequently showed that in characteristic $0$,
$\ed_k(\Sym_n) \geqslant \lfloor (n+1)/2 \rfloor$ for any $n \geqslant 7$.
% The best know upper bound is
% $\ed_k(\Sym_n) \leqslant n - 3$ for any field $k$; see~\cite[Theorem 6.5(c)]{br}.
The exact value of $\ed_k(\Sym_n)$ is open for every $n \geqslant 8$ and every field $k$, though 
it is widely believed that $\ed_k(\Sym_n)$ should be $n-3$ for every $n \geqslant 5$, 
at least in characteristic $0$.

The purpose of this paper is to show that $\ed_k(\Sym_n)$ can be $\leqslant n - 4$ in prime characteristic.
Our main result is as follows.

\begin{theorem} \label{thm.main}
Let $k$ be a field of odd characteristic $p > 0$ and $n$ is a positive integer whose binary presentation is
$n = 2^{m_1} + 2^{m_2} + \ldots + 2^{m_r}$, where $m_1 > m_2 > \ldots > m_r \geqslant 0$. Assume that $p$ divides $n$ and
$r \geqslant 4$. If $r = 4$, assume further that $k$ contains $\mathbb F_{p^2}$, a field of $p^2$ elements.
Then $\ed_k(\Sym_n) \leqslant n - 4$.
\end{theorem}

\begin{remark} \label{rem.main}
(a) Note that $\ed_k(\Sym_5) = 2$ and $\ed_k(\Sym_6) = 3$ for any field $k$ of characteristic $\neq 2$.
Thus, while it may be possible to weaken the assumption that $r \geqslant 4$ in the statement of Theorem~\ref{thm.main}, 
this assumption cannot be dropped entirely.

\smallskip
(b) If $k \subset k'$ is a field extension, then $\ed_k(\Sym_n) \geqslant \ed_{k'}(\Sym_n)$.
In particular, the conclusion of Theorem~\ref{thm.main} can be rephrased as follows: 
$\ed_{\mathbb F_p}(\Sym_n) \leqslant n - 4$ if $r \geqslant 5$
and $\ed_{\mathbb F_{p^2}}(\Sym_n) \leqslant n - 4$ if $r = 4$.

\smallskip
(c) Suppose $l$ is a field of characteristic $0$ and $k$ is a field of characteristic $p > 0$ containing an algebraic closure 
of $\mathbb F_p$. Then $\ed_l(\Sym_n) \geqslant \ed_k(\Sym_n)$ for any $n \geqslant 5$; see~\cite[Corollary 3.4(b)]{brv-mixed}.

\smallskip
(d) The smallest $n$ with $r \geqslant 4$ is $n = 8 + 4 + 2 + 1 = 15$.
\end{remark}

The remainder of this paper will be devoted to proving Theorem~\ref{thm.main}. In Section~\ref{sect.prelim} we will collect the 
background material on essential dimension that will be used in the proof. In Section~\ref{sect.prel12} we will introduce the $\Sym_n$-invariant subvariety
$X_{1, 2}$ of $\mathbb A^n$. In Section~\ref{sect.X12} we will 
show that under the assumptions of Theorem~\ref{thm.main}, the $\Sym_n$-action on $X_{1, 2}$
has maximal possible essential dimension: $\ed_k(X_{1, 2}; \Sym_n) = \ed_k(\Sym_n)$. 
Finally, in Section~\ref{sect.conclusion} we will complete the proof 
of Theorem~\ref{thm.main} by showing that $\ed_k(X_{1, 2}, \Sym_n) \leqslant n-4$.

\section{Preliminaries on essential dimension}
\label{sect.prelim}

Throughout this paper $k$ will denote an arbitrary base field, $\overline{k}$ will denote an algebraic closure of $k$, and $G$ will denote an abstract finite group. Unless otherwise specified,  algebraic varieties, morphisms, rational maps, group actions, etc., will be assumed to be defined over $k$. We will refer to a variety $X$ with an action of $G$ as a $G$-variety.  
We will say that the $G$-action on $X$ (or equivalently, the $G$-variety $X$) is 

\begin{itemize}
\smallskip \item
faithful, if the induced morphism $G \to \Aut(X)$ is injective,

\smallskip \item
primitive, if $G$ transitively permutes the irreducible components of $X(\overline{k})$,

\smallskip \item
generically free, if there exists a dense open subset $U \subset X$ such that for every $\overline{k}$-point $u \in U$, the stabilizer
$\Stab_G(u)$ of $u$ in $G$ is trivial.
\end{itemize}

A generically free action is clearly faithful. The converse holds if $X$ is irreducible, but not in general (not even if the $G$-action on $X$ is primitive).

Let $X$ be a generically free primitive $G$-variety. We will refer to a $G$-equivariant dominant rational map $X \dasharrow Y$ as a $G$-{\em compression}, if the $G$-action on $Y$ is also generically free.  The minimal dimension
of $Y$, taken over all $G$-compressions $X \dasharrow Y$ is called the {\em essential dimension} of $X$ and is denoted by $\ed_k(X; G)$.
The largest value of $\ed_k(X; G)$, as $X$ ranges over all generically free primitive $G$-varieties defined over $k$, is called the essential dimension of $G$ over $k$ and is denoted by $\ed_k(G)$.

We now recall two results about essential dimension that will be needed in the proof of Theorem~\ref{thm.main}. Note that Proposition~\ref{prop.prelim1} shows, in particular, that $\ed_k(G) < \infty$ for any $G$ and $k$.

\begin{proposition} \label{prop.prelim1} Let $G \hookrightarrow \GL(V)$ be a faithful finite-dimensional representation of
$G$. Denote the underlying affine space by $\mathbb A(V)$. Then
$\ed_k(G) = \ed_k(\mathbb A(V); G)$.
\end{proposition}

\begin{proof} See~\cite[Theorem 3.1]{br} or~\cite[Proposition 4.11]{berhuy-favi} or~\cite[Propositions 3.1 and 3.11]{merkurjev-survey}.
\end{proof}

Following~\cite{brv-mixed}, we will say that a finite group $G$ is weakly tame over a field $k$ of characteristic $p \geqslant 0$ if
$G$ has no normal $p$-subgroups, other than $\{ 1 \}$. If $p = 0$ (or if $p$ does not divide $|G|$),
then $G$ is always weakly tame over $k$.

\begin{proposition} \label{prop.prelim3} Suppose that a finite group $G$ is weakly tame group over a field $k$.
Let $X$ and $Y$ be generically free primitive $G$-varieties over $k$. Assume that there exists 
a (not necessarily dominant) $G$-equivariant rational map $f \colon Y \dasharrow X^{\rm sm}$, where 
$X^{\rm sm}$ denotes the smooth locus of $Y$. Then $\ed_k(X) \geqslant \ed_k(Y)$.
\end{proposition}

\begin{proof} % (a) 
See~\cite[Theorem 1.6]{reichstein2022behavior}. 
\end{proof}

\section{Preliminaries on the affine quadric $X_{1, 2}$}
\label{sect.prel12}

Let $X_{1, 2}$ be the closed $\Sym_n$-invariant subvariety of $\mathbb A^n$ given by
\[ s_1(x_1, \ldots, x_n) = s_2(x_1, \ldots, x_n) = 0 , \]
where $s_i$ is the $i$th elementary symmetric polynomial and $\Sym_n$ acts on $\mathbb A^n$ 
by permuting the variables in the natural way. If $\Char(k) \neq 2$, then equivalently,
\begin{equation} \label{e.X12}
X_{1, 2} := \{ \, (x_1, \ldots, x_n) \in \mathbb A^n \, | \, x_1+ \ldots + x_n = x_1^2 + \ldots + x_{n}^2 = 0 \, \} . 
\end{equation}
Let $\Delta$ be the discriminant locus in $\mathbb A^n$, i.e., the union of the hyperplanes $x_i = x_j$ for various pairs $(i, j)$, where $1 \leqslant i < j \leqslant n$. Note that the symmetric group $\Sym_n$ acts freely (i.e., with trivial stabilizers) 
on $\mathbb A^n \setminus \Delta$.

\begin{lemma} \label{lem.X12} Assume $n \geqslant 5$. Then

\smallskip
(a) $X_{1, 2}$ is absolutely irreducible.

\smallskip
(b) $X_{1,2} \setminus \Delta$ is a dense open subset of $X_{1, 2}$. In particular, the $\Sym_n$-action on $X_{1, 2}$ is generically free.

\smallskip
(c) The singular locus of $X_{1, 2}$ is $X_{1, 2} \cap D$, where $D$ is the small diagonal $D$ given by $x_1 = \ldots = x_n$.
\end{lemma}

\begin{proof} By definition $X_{1, 2}$ is the affine cone over the projective variety
$\mathbb P(X_{1, 2}) \subset \mathbb P^{n-1}$ given by 
\[ x_1 + \ldots + x_n = x_1^2 + \ldots + x_n^2 = 0 \, . \]
Parts (a), (b) and (c) of the Lemma now follow from~\cite[Lemma 2.1]{brassil-reichstein}, which shows
that $\mathbb P(X_{1, 2})$ is absolutely irreducible (Lemma 2.1(c)),
$\mathbb P(X_{1, 2})$ is not contained in $\mathbb P(\Delta)$ (Lemma 2.1(f)), and the singular locus of $\mathbb P(X_{1, 2})$ is $\mathbb P(X_{1, 2}) \cap \{ (1: 1: \ldots : 1)\}$ (Lemma 2.1(b)). Note in \cite{brassil-reichstein}, $\mathbb P(X_{1, 2})$ is denoted by $Y_{n, 2}$ and $\Delta$ is denoted by $\Delta_n$. 
\end{proof}

\section{Reduction to $X_{1, 2}$}
\label{sect.X12}
The purpose of this section is to prove the following.

\begin{proposition} \label{prop.X12} Assume $k$ and $n$ are as in the statement of Theorem~\ref{thm.main}. 
Then \[ \ed_k (X_{1, 2}; \Sym_n) = \ed_k(\Sym_n). \]
\end{proposition}

% Let $K_n = k(a_1, \ldots, a_n)$, where $a_1, \ldots, a_n$ be independent variables and $L_n$ be an extension of $K_n$
% obtained by adjoining a root of the "general polynomial" $f(x) = x^n + a_1 x^{n-1} + \ldots + a_n$ of degree $n$. Note that
% $f(x)$ is irreducible over $K_n$ by the Eisentstein criterion. Hence, $L_n = K_n[x]/(f(x))$.
Note that $\ed(X_{1, 2}; \Sym_n)$ is well defined because $X_{1, 2}$ is an absolutely irreducible generically free 
$\Sym_n$-variety by Lemma~\ref{lem.X12}. Lemma~\ref{lem.X12} applies here because our assumptions 
force $n$ to be $\geqslant 15$; see Remark~\ref{rem.main}(d).

Our proof of Proposition~\ref{prop.X12} will be based on the following.

\begin{lemma} \label{lem.trace12} Assume $k$ and $n$ are as in the statement of Theorem~\ref{thm.main}. 
Let $F$ be a field containing $k$ and $E/F$ be an $n$-dimensional \'etale algebra. Then there exists an
element $\alpha \in E \setminus F$ such that $\Tr(\alpha) = \Tr(\alpha^2) = 0$. Here $\Tr$ denotes the trace function $\Tr_{E/F} \colon E \to F$.
\end{lemma}

\begin{proof} Let $E^{\rm{trace0}}$ be the kernel of the trace map $\Tr \colon E \to F$. It is an $(n-1)$-dimensional 
$F$-vector subspace of $E$. Since $n$ is divisible by $p = \Char(k) = \Char(F)$, $F$ is contained in $E^{\rm{trace0}}$. Consider the symmetric
bilinear form $\psi \colon E^{\rm{trace0}} \to F$ given by $\psi(a, b) = \Tr(ab)$ and the associated quadratic form $q \colon E^{\rm{trace0}} \to F$ given by 
$q(a) = \psi(a, a) = \Tr(a^2)$. By the definition of $E^{\rm{trace0}}$, $\psi(1, a) = \Tr(a) = 0$ for every $a \in E^{\rm{trace0}}$. Thus $F = F \cdot 1$ lies in the radical of $q \colon E^{\rm{trace0}} \to F$. Consequently, the bilinear form $\psi$ and the quadratic form $q$
descend to a bilinear form $\overline{\psi}$ and a quadratic form $\overline{q}$ on the $(n-2)$-dimensional quotient space 
$\overline{E^{\rm{trace0}}} = E^{\rm{trace0}}/(F \cdot 1)$. 

Elements $\alpha \in E$ such that $\Tr(\alpha) = \Tr(\alpha^2) = 0$ are precisely the isotropic vectors of $q$ in $E^{\rm{trace0}}$.
We are looking for an element $\alpha \in E^{\rm{trace0}} \setminus F$ whose image in $\overline{E^{\rm{trace0}}}$ is an isotropic vector for $\overline{q}$. 
In other words, the lemma is equivalent to the assertion that the quadratic form $\overline{q} \colon \overline{E^{\rm{trace0}}} \to F$ is isotropic.

To show that $\overline{q}$ is isotropic, we will appeal to Springer's theorem: 
If $F'/F$ is a field extension of odd degree, then $\overline{q}$ is isotropic in $\overline{E^{\rm{trace0}}}$ if and only 
if it becomes isotropic over $F'$.
By \cite[Proposition 5.1]{brassil-reichstein} we can choose $F'/F$ so that $E' = E \otimes_{F} F'$ is an \'etale algebra 
of degree $n$ over $F'$ the form $E' = E_1 \times E_2 \times \ldots \times E_r$, where $E_i$ is an \'etale algebra 
over $F'$ of degree $2^{m_i}$ for each $i = 1, \ldots, r$. After replacing $F$ by $F'$ and $E$ by $E'$, we may assume 
without loss of generality that, in fact,
\[ E = E_1 \times E_2 \times \ldots \times E_r, \]
where $E_i$ is an $2^{m_i}$-dimensional \'etale algebra over $F$ for each $i = 1, \ldots, r$. Note that 
% the trace form $\Tr_{E/F} \colon E \to F$ is given by
\[ \Tr_{E/F}(a) = \Tr_{E_1/F}(a_1) + \ldots + \Tr_{E_r/F}(a_r) = 0  \]
for any $a = (a_1, \ldots, a_r) \in E_1 \times \ldots \times E_r = E$.
In particular, we now have an $r$-dimensional $F$-subalgebra of $E$,  
\[ F \times \ldots \times F \; \text{($r$ times)} \; \subset E_1 \times \ldots \times E_r = E , \]
where the trace form is quite transparent: $\Tr_{E/F} \colon (c_1, \ldots, c_r) \mapsto 2^{m_1} c_1 + \ldots + 2^{m_r} c_r$ for any 
$c_1, \ldots, c_r \in F$. Specifically, we would like to show that $\overline{q}$ has an isotropic vector in $V/F$, where 
\[ V = \{  (c_1, \ldots, c_r) \in F \times \ldots \times F \, | \, \Tr_{E/F}(c_1, \ldots, c_r) = 0 \}. \]
Equivalently, we would like to show that $q$ has an isotropic vector in $V \cap H$, where $H$ is a $F$-hyperplane in $F \times \ldots \times F$ ($r$ times)
which does not contain the unit element $(1, \ldots, 1)$. For example, we can take $H$ to be the hyperplane $c_r = 0$.
Explicitly, we are looking for a non-trivial solution to the system
\begin{equation} \label{e.system}
\begin{cases} 2^{m_1} c_1 + \ldots + 2^{m_{r-1}} c_{r-1} = 0 \\ 2^{m_1} c_1^2 + \ldots + 2^{m_{r-1}} c_{r-1}^2 = 0
\end{cases} 
\end{equation}
Note that all the coefficients in this system are integers. 
By a theorem of Chevalley~\cite[Theorem 5.2.1]{pfister}, 
every finite field is a $C_1$-field, and consequently, the system~\eqref{e.system} of one linear and one quadratic equations
in $r-1$ variables has a non-trivial solution over $\mathbb F_p$, as long as $r - 1 > 3$. This completes the proof of 
Lemma~\ref{lem.trace12} for $r \geqslant 5$.

 If $r = 4$, then we may or may not be able to find a non-trivial solution of the system~\eqref{e.system} over $\mathbb F_p$, but there will certainly be one over some quadratic extension of $\mathbb F_p$. Since $\mathbb F_p$ has a unique quadratic extension, $\mathbb F_{p^2}$, and we are assuming that
 $k$ contains a copy of $\mathbb F_{p^2}$ , we conclude that in this case the system~\eqref{e.system}
 has an isotropic vector over $k$ and hence, over $F$. This completes the proof of Lemma~\ref{lem.trace12} for $r = 4$.
 \end{proof}

In the proof of Proposition~\ref{prop.X12} below, we will apply Lemma~\ref{lem.trace12}
to the general field extension $L_n/K_n$ defined as follows: $K_n = k(a_1, \ldots, a_n)$, 
where $a_1, \ldots, a_n$ be independent variables and $L_n$ be an extension of $K_n$
obtained by adjoining a root of the "general polynomial" $f(x) = x^n + a_1 x^{n-1} + \ldots + a_n$ of degree $n$. Note that
$f(x)$ is irreducible over $K_n$ by the Eisentstein criterion. Hence, $L_n = K_n[x]/(f(x))$. 

\begin{remark} \label{rem.good-char} 
Lemma~\ref{lem.trace12} may be viewed as a ``bad characteristic variant" of~\cite[Corollary 10.1(c)]{brassil-reichstein}.
Note that under the assumption that $\Char(k)$ divides $n$, Corollary 10.1(c) becomes vacuous. The above proof 
is an enhanced version of the argument from~\cite{brassil-reichstein}, under a somewhat stronger assumption.
Theorem 1.4 (with $p= 2$) and Corollary 10.1(c) in \cite{brassil-reichstein} 
require the system
\[
\begin{cases} 2^{m_1} c_1 + \ldots + 2^{m_{r}} c_{r} = 0 \\ 2^{m_1} c_1^2 + \ldots + 2^{m_{r}} c_{r}^2 = 0
\end{cases} 
\]
to have a non-trivial solution in $k$. 
When $k$ is algebraically closed (and $\Char(k)$ is different from $2$ and does not divide $n$), 
this is equivalent to $r \geqslant 3$.
Here we ask for a solution with $c_r = 0$; see~\eqref{e.system}. This requires $r$ to be $\geqslant 4$.
\end{remark}

\begin{remark} If $F$ is an infinite field, one can always choose $\alpha$ in Lemma~\ref{lem.trace12} so that
it generates $E$ over $F$, i.e., $F[\alpha] = E$. We will not need this assertion here; we refer an interested 
reader to Assertion $(**)$ in~\cite[Theorem 1.4]{brassil-reichstein} and its proof in \cite[Section 8]{brassil-reichstein}. 

Note also that we will only use Lemma~\ref{lem.trace12} in the special case, where $E/F$ is the general field extension $L_n/F_n$ 
defined above. In this case there are no intermediate fields, strictly between $L_n$ and $K_n$, so any $\alpha \in E \setminus F$ 
is automatically a generator.
% 
% For an arbitrary \'etale algebra $L$ of degree $n$ over an infinite field $F$, we argue as follows. 
% Let $Q$ be the quadric hypersurface in the $n-3$-dimensional
% projective space $\mathbb P(L^{\rm{trace0}}/K)$. It follows from Lemma~\ref{lem.X12}(c) that $Q$ is smooth.
% The element $\alpha$ gives rise to a $F$-point on $Q$. Stereographic projection from this point shows 
% that $Q$ is, in fact, rational over $F$ and hence, $F$-points are dense in $Q$. On the other hand, generators 
% of $E/F$ correspond to $F$-points in an open subset of $\mathbb P(L^{\rm{trace0}}/K)$. It follows from Lemma~\ref{lem.X12}(b) 
% that this subset is non-empty. Hence, this open subset has a $F$-point. This point corresponds to an element $\alpha$
% such that $\Tr(\alpha) = \Tr(\alpha^2) = 0$ and $K[\alpha] = L$. \qed
\end{remark} 

\begin{proof}[Proof of Proposition~\ref{prop.X12}] 
% Let $\alpha \in L_n$ be as in Lemma~\ref{lem.trace12}. Write
% $\alpha = c_0 + c_1 x + \ldots + c_{n-1} x^{n-1}$, where each $c_i = c_i(a_1, \ldots, a_n) \in K_n$ is a rational function in
% $a_1, \ldots, a_n$ with coefficients in $k$. 
Denote the roots of the general polynomial $f(x) = x^n + a_1 x^{n-1} + \ldots + a_0$ by
$x_1, \ldots, x_n$. 
Then $a_i = (-1)^i s_i(x_1, \ldots, x_n)$, where $s_i$ denotes the $i$th symmetric polynomial.
Since $a_1, \ldots, a_n$ are algebraically independent over $k$, so are $x_1, \ldots, x_n$.
Identify $K_n$ with $k(x_1, \ldots, x_n)^{\Sym_n}$ and $L_n$ with $K_n(x_1) = k(x_1, \ldots, x_n)^{\Sym_{n-1}}$, where
$\Sym_n$ naturally permutes $x_1, \ldots, x_n$ and $\Sym_{n-1}$ is the stabilizer of $x_1$ in $\Sym_n$.

It is well known that elements of $L_n$ are in bijective correspondence with 
$\Sym_n$-equivariant rational maps $\phi \colon \mathbb A^n \dasharrow \mathbb A^n$. Indeed,
write 
$\phi(x_1, \ldots, x_n) = \big( \phi_1(x_1, \ldots, x_n), \ldots, \phi_n(x_1, \ldots, x_n) \big)$.
A priori the components $\phi_1, \ldots, \phi_n$ of $\phi$ lie in $k(x_1, \ldots, x_n)$; however,
since $\phi$ is $\Sym_n$-equivariant, $\phi_1$ actually lies in $k(x_1, \ldots, x_n)^{\Sym_n} = L_n$.
Conversely, given $\alpha \in L_n$, we can define $\phi \colon \mathbb A^n \dasharrow \mathbb A^n$
by 
\begin{equation} \label{e.phi}
\phi(x) = \big( \alpha_1(x_1, \ldots, x_n),  \ldots, \alpha_n(x_1, \ldots, x_n) \big) , 
\end{equation}
where $\alpha_1, \ldots, \alpha_n$ are the $\Sym_n$-translates of $\alpha = \alpha_n \in L_n$. Note that there are exactly 
$n$ distinct $\Sym_n$-translates if $\alpha$ does not lie in $K_n$. If $\alpha$ lies in $K_n$, then $\alpha_1 = \ldots = \alpha_n$.

Now choose $\alpha \in L_n$ as in Lemma~\ref{lem.trace12}, and let $\phi \colon \mathbb A^n \dasharrow \mathbb A^n$
be the rational $\Sym_n$-equivariant map given by~\eqref{e.phi}.
Since $s_1(\alpha) = s_2(\alpha) = 0$, the image of $\phi$ is contained in $X_{1, 2} \subset \mathbb A^n$. 
Since $\alpha \not \in K_n$, the general point of $\mathbb A^n$ maps to $X_{1, 2} \setminus D$,
where  $D$ is the small diagonal in $\mathbb A^n$ given by $x_1 = \ldots = x_n$,
as in Lemma~\ref{lem.X12}(c).
In other words, we may think of $\phi$ as a $\Sym_n$-equivariant map $\mathbb A^n \dasharrow X_{1, 2} \setminus D$.
Now recall that since $\mathbb A^n$ is an affine space with a linear action of $\Sym_n$, 
\begin{equation} \label{e.prop.X12-1} \ed_k(\Sym_n) = \ed_k(\mathbb A^n; \Sym_n) \geqslant \ed_k(X_{1,2}; \Sym_n) ; 
\end{equation}
see~Proposition~\ref{prop.prelim1}. On the other hand, by Proposition~\ref{prop.prelim3},
% \cite[Theorem 1.6]{reichstein2022behavior},  
\begin{equation} \label{e.prop.X12-2}
\ed_k(\mathbb A^n; \Sym_n) \leqslant \ed_k(X_{1,2}; \, \Sym_n). 
\end{equation}
Proposition~\ref{prop.prelim3} applies here because $X_{1, 2}$ is an irreducible generically free $\Sym_n$-variety,
$X_{1, 2} \setminus D$ is smooth (see Lemma~\ref{lem.X12}),
and the symmetric group  $\Sym_n$ is weakly tame at any prime. (Once again, here $n \geqslant 15$; see Remark~\ref{rem.main}(d).) 

Combining~\eqref{e.prop.X12-1} and~\eqref{e.prop.X12-2}, we obtain
the desired equality, $\ed_k(\Sym_n) = \ed_k(X_{1,2}; \Sym_n)$.
\end{proof}

\section{Conclusion of the proof of Theorem~\ref{thm.main}}
\label{sect.conclusion}

In this section we will complete the proof of Theorem~\ref{thm.main} by establishing the following.

\begin{proposition} \label{prop.conclusion-main}
Let $k$ be a base field of characteristic $p > 2$.  Assume that $n \geqslant 5$ and $p$ divides $n$.
Then $\ed_k(X_{1,2}; \Sym_n) \leqslant n- 4$.  
\end{proposition} 

Note that the assumption of Theorem~\ref{thm.main} that $n$ should not be a sum of three or fewer powers of $2$, is not needed here.

Before proceeding with the proof of Proposition~\ref{prop.conclusion-main}, we briefly outline our overall strategy. 
Our goal is to show the existence of a $\Sym_n$-compression $\pi \colon X_{1, 2} \dasharrow Y$ defined over $k$, 
where $\dim(Y) \leqslant n- 4$. Key to our construction is the observation that 
$X_{1, 2}$ admits an action of a $2$-dimensional linear algebraic group $B$, 
which commutes with the $\Sym_n$-action. This $B$-action is a characteristic $p$ phenomenon; 
it only exists when $\Char(k)$ divides $n$. The idea is then to define $\pi$ as the quotient map for this action.
The remainder of this section will be devoted to working out the details of this construction.

We begin by introducing the $2$-dimensional algebraic group $B$. It is the group of upper-triangular matrices in $\PGL_2$, 
i.e., the group of matrices of the form $\begin{pmatrix} \alpha & \beta \\ 0 & 1 \end{pmatrix}$. 
(Here $B$ stands for ``Borel subgroup".) Consider the natural action of $B$ on $\mathbb A^n$ by 
\[ \begin{pmatrix} \alpha & \beta \\ 0 & 1 \end{pmatrix} \cdot (x_1, \ldots, x_n) \to 
(\alpha x_1 + \beta, \ldots, \alpha x_n + \beta) . \]

\begin{lemma} \label{lem.conclusion1}
(a) $X_{1, 2} \subset \mathbb A^n$ is invariant under the action of $B$ defined above.

(b) The stabilizer in $B$ of a point $a = (a_1, \ldots, a_n) \in X_{1, 2} \setminus \Delta$ is trivial. 
\end{lemma}

\begin{proof} (a) We need to check that 
$\Tr(a) = \Tr(a^2) = 0$ implies $\Tr(g \cdot a) = \Tr((g \cdot a)^2) = 0$,
for any $a = (a_1, \ldots, a_n) \in \mathbb A^n(\overline{k})$
and any $g = \begin{pmatrix} \alpha &  \beta \\
0 & 1
\end{pmatrix} \in B$. Here 
% $\overline{k}$ denotes the algebraic closure of $k$ and
$\Tr$ denotes the trace in
the split $\overline{k}$-\'etale algebra $\overline{k} \times \ldots \times \overline{k}$ ($n$ times), i.e.,
$\Tr(a) = a_1 + \ldots + a_n$ and
$\Tr(a^2) = a_1^2 + \ldots + a_n^2$.

Indeed, assume that $\Tr(a) = \Tr(a^2) = 0$. Since $\Char(k)$ divides $n$, we have
\[ \Tr(g \cdot a) = \Tr (\alpha a_1 + \beta, \ldots, \alpha a_n + \beta) = \alpha \Tr(a) + n \beta = 0 + 0 = 0 \] 
and
% \begin{align*}
$ \Tr((g \cdot a)^2)  = \Tr((\alpha a_1 + \beta)^2, \ldots, (\alpha a_n + \beta)^2) 
% \\  &   
   = \alpha^2 \Tr(a^2) + 2 \alpha \beta \Tr(a) + n \beta^2  = 0 + 0 + 0 = 0$, 
% \, ,
% \end{align*}
as desired.

\smallskip
(b) The stabilizer of any point $a = (a_1, \ldots, a_n) \in \mathbb A^n(\overline{k})$ is the group subscheme of $B$ cut out by the equations
\begin{equation} \label{e.conclusion1}
\begin{cases} 
\alpha a_1 + \beta = a_1, \\
\ldots \\
\alpha a_n + \beta = a_n.
\end{cases}
\end{equation}
Here $a_1, \ldots, a_n \in \overline{k}$ are fixed, and $\alpha$ and $\beta$ are coordinate functions on $B$.
Rewriting this system in matrix form, we obtain 
\[ (\alpha - 1, \, \beta) \cdot \begin{pmatrix} a_1  & a_2 & \ldots & a_n  \\
                                             1 & 1 & \ldots & 1 \end{pmatrix} = \begin{pmatrix} 0 &  \ldots & 0 \end{pmatrix}. \]
If $a \not \in D\subset \Delta$, i.e., at least two of the elements
$a_1, \ldots, a_n$ of $k$ are distinct, then the $2 \times n$ matrix 
$\begin{pmatrix} a_1  & a_2 & \ldots & a_n  \\
1 & 1 & \ldots & 1 \end{pmatrix}$ 
has rank $2$. Hence, the kernel of this matrix is trivial. We conclude that (scheme-theoretic) solution set to the system~\eqref{e.conclusion1} consists of a single point, $(\alpha, \beta) = (1, 0)$. In other words,
the stabilizer of $a$ in $B$ is trivial.                                        
\end{proof}

We now define the $\Sym_n$-equivariant morphism $\pi \colon (X_{1, 2} \setminus \Delta) \to \mathbb A^{n(n-1)(n-2)}$ by
\[ \pi \colon a = (x_1, \ldots, x_n) \mapsto \Big( \,  \dfrac{x_{r} - x_{s}}{x_r - x_t} \, \Big)_{(r, s, t)} \]
where the subscript $(r, s, t)$ ranges over the $n(n-1)(n-2)$ ordered triples of distinct integers in $\{ 1, 2, \ldots, n \}$, 
and $\Sym_n$ acts on these triples in the natural way. Clearly, each $\dfrac{x_{r} - x_{s}}{x_r - x_t}$
is a regular function on $X_{1, 2} \setminus \Delta$. Letting $Y$ be the Zariski closure of the image of $\pi$ in $\mathbb A^{n(n-1)(n-2)}$, we may view $\pi$ as an $\Sym_n$-equivariant dominant 
rational $X_{1, 2} \dasharrow Y$.
The following lemma completes the proof of Proposition~\ref{prop.conclusion-main}.

\begin{lemma} \label{lem.conclusion2} Assume that $n \geqslant 5$. Then
(a) $\Sym_n$ acts faithfully on $Y$, and
(b) $\dim(Y) \leqslant n - 4$.
\end{lemma}

\begin{proof} (a) Assume the contrary. The kernel $N$ of this action is a non-trivial normal subgroup of $\Sym_n$.
Since $n \geqslant 5$, $N$ is either the alternating group $\Alt_n$ or the full symmetric group $\Sym_n$. In both cases
$\Alt_n$ acts trivially on $\pi(a)$ for every $a = (a_1, \ldots, a_n) \in X_{1, 2} \setminus \Delta$. In particular,
the 3-cycle $\sigma = (2, \, 4, \, 5) \in \Alt_n$ preserves $\pi(a)$. 
% (Here we use the assumption that $n \geqslant 5$.) 
That is,
\begin{equation} \label{e.pi}
\dfrac{a_{1} - a_{2}}{a_1 - a_3} = \sigma \cdot  \dfrac{a_{1} - a_{2}}{a_1 - a_3} = \dfrac{a_{1} - a_{4}}{a_1 - a_3} . 
\end{equation}
This implies $a_2 = a_4$, which contradicts our assumption that $a \not \in \Delta$.

\smallskip
(b) Note that $\pi$ sends every $B$-orbit to a point. By Lemma~\ref{lem.conclusion1}(b), a general orbit of $B$ in $X_{1, 2}$ is 2-dimensional. Hence, a general fiber of $\pi$ is of dimension $\geqslant 2$. By the Fiber Dimension Theorem, $\dim(Y) \leqslant \dim(X_{1, 2}) - 2 = n- 4$.
\end{proof}

\begin{remark} \label{rem.conclusion1}
As we suggested at the beginning of this section, $\pi$ is, in fact, a rational quotient for the $B$-action on $X_{1, 2}$. We do not need to know this though; the explicit formula in~\eqref{e.pi} suffices for the purpose of proving Proposition~\ref{prop.conclusion-main}.
\end{remark}

\bibliographystyle{plain}

\end{document}